\documentclass{amsart}
\usepackage[margin=3.5cm]{geometry}
\usepackage{xcolor}
\usepackage[T1]{fontenc}
\usepackage{amsfonts, amsthm, amssymb,verbatim,amscd,amsmath, blindtext}

\usepackage{multirow}
\usepackage{graphicx}
\usepackage{enumitem,yhmath}
\usepackage{amssymb}
\usepackage{ytableau}
\usepackage{tikz}
\usepackage{tikz-cd}
\usepackage{float, pifont}
\usepackage{mathtools}
\usepackage[pagebackref,colorlinks=true,citecolor=blue,linkcolor=black]{hyperref}

\newtheorem{theorem}{Theorem}[section]
\newtheorem{lemma}[theorem]{Lemma}
\newtheorem{corollary}[theorem]{Corollary}

\newtheorem{conjecture}[theorem]{Conjecture}
\newtheorem{question}[theorem]{Question}

\theoremstyle{remark}

\theoremstyle{definition}

\newtheorem{definition}[theorem]{Definition}

\DeclareMathOperator{\characteristic}{char}
\DeclareMathOperator{\perm}{perm}
\DeclareMathOperator{\initial}{in}

\numberwithin{equation}{section}

\begin{document}

\title{The $F$-singularities of algebras defined by permanents}

\author{Trung Chau}
\address{Chennai Mathematical Institute, Siruseri, Tamil Nadu, India}
\email{chauchitrung1996@gmail.com}

\keywords{permanent, permanental ideal, $F$-purity, $F$-regularity, generic matrix, symmetric matrix, Hankel matrix, determinantal ideal}

\subjclass[2020]{13A35; 14G17; 14M99; 13C40; 15A15}

\begin{abstract}
    Let $X$ be a matrix of indeterminates, $t$ an integer, and $P_t(X)$ define the ideal generated by the permanents of all $t\times t$ submatrix of $X$. $P_t(X)$ is called a permanental ideal. In this article, we study the algebras $\Bbbk[X]/P_t(X)$ where $X$ is a generic, symmetric, or a Hankel matrix of indeterminates. When $\characteristic \Bbbk = 2$, $P_t(X)$ is also known as a determinantal ideal, a popular class in commutative algebra and algebraic geometry, and thus many properties of $P_t(X)$ are known in this case. We prove that, if $X$ is an $n\times n$ matrix and $\characteristic \Bbbk>2$, the algebra $\Bbbk[X]/P_n(X)$ is $F$-regular, just like when $\characteristic \Bbbk = 2$. On the other hand, we obtain a full characterization of when $\Bbbk[X]/P_2(X)$ is $F$-pure or $F$-regular, when $\characteristic \Bbbk >2$, and the answer is different than that in even characteristic. 
\end{abstract}
\maketitle

\section{Introduction}

Let $X=(x_{ij})_{1\leq i,j\leq n}$ be an $n\times n$ matrix for some integer $n$. The \emph{permanent} and \emph{determinant} of $X$, denoted by $\perm(X)$ and $\operatorname{det}(X)$, are defined to be the polynomials
\[
\sum_{\sigma} x_{1,\sigma(1)}\cdots x_{n,\sigma(n)} \text{ and }\sum_{\sigma} (-1)^{\operatorname{sgn}(\sigma)}x_{1,\sigma(1)}\cdots x_{n,\sigma(n)} ,
\]
respectively, where  the summation ranges over all the permutations of the set $\{1,\dots, n\}$. Permanents and determinants are among the most important polynomials in mathematics. While determinants need no introduction, permanents, since their invention by Cauchy and Binet in the early 1800s, have found applications across mathematics and physics, e.g., counting problems regarding matchings in Graph Theory \cite{PermanentProblem} or the study of boson Green's function in Quantum Field Theory \cite{Aaronson:14}, to name a few. We refer to \cite{Minc} for an in-depth survey on the subject. Perhaps  it is the work of Valiant \cite{VALIANT1979189} that gives the most compelling reason why permanents are important: their computation is an NP-hard problem. Thus an algorithm that computes permanents in polynomial time would imply $P=NP$, solving the arguably most well-known question in theoretical computer~science. On the other hand, determinants are known to be computable in polynomial time using Gaussian elimination. This marks one of the most important difference in the study of permanents, despite their structural similarity to determinants.

Let $\Bbbk$ be a field and $X$ be a matrix of indeterminates. For each integer $t$, we define the \emph{$t\times t$ permanental ideal} (respectively, \emph{determinantal ideal}) of $\Bbbk[X]$, denoted by $P_t(X)$ (respectively, $I_t(X)$), to be the one generated by all the permanents (respectively, determinants) of all the $t\times t$ submatrices of $X$. Determinantal ideals are among the most popular objects in Commutative Algebra and Algebraic Geometry, as they appear naturally in the study of these subjects, e.g., they define invariant rings of linearly reductive group actions on polynomial rings, or the Grassmanian varieties. Three cases of special interests are when $X$ is \emph{generic}, \emph{symmetric} ($x_{ij}=x_{ji}$ for any $i,j$), and \emph{Hankel} ($x_{ij}=x_{i+k,\ j-k}$ for any $i\leq j$ and any $k\in [0,j-i]$). Determinantal ideals in these cases define algebras that have rational singularities, and hence are Cohen-Macaulay normal domains, in characteristic $0$  \cite{Bou87, CONCA2018111, Eagon1971InvariantTA, Kutz74}.  In prime characteristic $p>0$, Hochster and Huneke \cite{HH-BrianconSkoda} initiated the study of $F$-singularities which has become an active area since, and determinantal ideals of generic and symmetric matrices define algebras that are among the first examples of $F$-regular rings, the closest singularities to being regular in the theory. It is conjectured in \cite{CONCA2018111} that the same holds for Hankel matrices, and still remains open. We refer to \cite{BCRV, BV} for some survey on determinantal ideals.

Permanental ideals, on the algebraic side, have remarkably different properties than their determinantal counterparts. It is important to note that in characteristic $2$, the two concepts coincide, and thus permanental ideals share the properties with determinantal ideals. The story becomes complicated in other characteristics. For example, while minimal primes of $P_2(X)$ are known in the cases where $X$ is generic \cite{LS00}, symmetric \cite{Trung-symmetric}, and Hankel \cite{GGS07}, a complete list of those of $P_3(X)$ remains elusive in the case where $X$ is generic, with a partial answer given by Kirkup \cite{Kirkup}. The algebras $\Bbbk[X]/P_2(X)$ are rarely reduced and Cohen-Macaulay, in all three cases for the matrix $X$, a noticeable deviation from the algebras defined by determinants. There are only few studies on $P_t(X)$ for any $t$ (see, e.g., \cite{BCMV25,ELSW18}). 

In this article, we assume $\characteristic k = p>2$, and study the $F$-singularities of the permanental hypersurface $\Bbbk[X]/\big(\perm(X)\big)$ when $X$ is a square generic/symmetric/Hankel matrices, and $\Bbbk[X]/\sqrt{P_2(X)}$ when $X$ is a generic or symmetric matrix of any size. Note that an $F$-pure ring is reduced, and $\Bbbk[X]/P_2(X)$ is rarely reduced, hence rarely $F$-pure. It is however, somewhat surprising that its reduction $\Bbbk[X]/\sqrt{P_2(X)}$ always satisfies the property. We refer to Section~\ref{sec:prem} for unexplained concepts. We obtain a full characterization of $F$-purity and $F$-regularity for these rings. Our main theorems are as follows.

\begin{theorem}[{Theorems~\ref{thm:square-Hankel-F-regular} and \ref{thm:square-generic-symmetric-F-regular}}]
    Let $\Bbbk$ be a field of positive characteristic $p>2$ and $X$ a square generic/symmetric/Hankel matrix of indeterminates. Then the ring $\Bbbk[X]/\big(\perm(X)\big)$ is $F$-regular.
\end{theorem}

\begin{theorem}[{Theorem~\ref{thm:generic-Fpure}}]
    Let $\Bbbk$ be a field of positive characteristic $p>2$ and $X$ an $m\times n$ generic  matrix of indeterminates where $m,n\geq 2$. Then
	\begin{enumerate}
		\item the ring $\Bbbk[X]/\sqrt{P_2(X)}$ is $F$-pure;
		\item the ring $\Bbbk[X]/\sqrt{P_2(X)}$ is $F$-regular if and only if $m=n=2$.
	\end{enumerate}
\end{theorem}

\begin{theorem}[{Theorem~\ref{thm:symmetric-Fpure}}]
    Let $\Bbbk$ be a field of positive characteristic $p>2$ and $Y$ an  $n\times n$  symmetric matrix of indeterminates where $n\geq 2$. Then
	\begin{enumerate}
		\item the ring $\Bbbk[Y]/\sqrt{P_2(Y)}$ is $F$-pure;
		\item the ring $\Bbbk[Y]/\sqrt{P_2(Y)}$ is $F$-regular if and only if $n=2$.
	\end{enumerate}
\end{theorem}

In this article we will often use $X$ a generic matrix, $Y$ a symmetric matrix, and $Z$ a Hankel matrix. We remark that in the case where $Z$ is a Hankel matrix, it is known that except for the hypersurface case, the ideal $\sqrt{P_2(Z)}$ is a squarefree monomial ideal  in odd characteristics \cite{GGS07}, and thus defines a non-$F$-regular $F$-pure ring (see, e.g., \cite[Exercise~3.35]{GIAN}). Thus an analog of our last two results for Hankel matrices follows immediately.

The article is structured as follows. Section~\ref{sec:prem} provides some background on $F$-singularities and permanental ideals. Section~\ref{sec:square} focuses on the $F$-singularities of the hypersurface defined by the permanent of a square generic/symmetric/Hankel matrix of indeterminates. The goal of Section~\ref{sec:generic} is the complete characterization of when $\Bbbk[X]/\sqrt{P_2(X)}$ is $F$-pure or $F$-regular when $X$ is a generic matrix, and Section~\ref{sec:sym} proves the analogs in the case when $X$ is symmetric.

\section{Preliminaries}\label{sec:prem}

\subsection{The $F$-purity and $F$-regularity of $\mathbb{N}$-graded rings}

Let $S$ denote the polynomial ring $\Bbbk[x_1,\dots, x_n]$ over a field $\Bbbk$ with standard grading, i.e., all variables are of degree $1$. Let $I$ be a homogeneous ideal of $S$, and $R=S/I$ denote its corresponding homomorphic image. We remark that these are examples of $\mathbb{N}$-graded rings. For the rest of the section, unless stated otherwise, we assume $\characteristic \Bbbk = p>0$. Moreover, $q=p^e$ always denotes a power of $p$ with $e$ a positive integer. 

For an ideal $J=(m_1,\dots, m_k)$ of $R$, we define the $q$-th \emph{Frobenius power} of $J$ to be
\[
J^{[q]}\coloneqq (m_1^q,\dots, m_k^q).
\]

We define the \emph{Frobenius closure} of $J$, denoted by $J^F$, to be the collection of elements $r\in R$ such that $r^{q} \in J^{[q]}$ for some $q=p^e$. We define the \emph{tight closure} of $J$, denoted by $J^*$, to be the collection of elements $r\in R$ such that $cr^{q} \in J^{[q]}$ for all sufficiently large $q=p^e$ and some non-zerodivisor $c$ of $R$. It is clear from the definitions that $J\subseteq J^F \subseteq J^*$.

\begin{definition}
    The ring $R$ is called \emph{$F$-regular} (respectively, \emph{$F$-pure}) if $J=J^*$ (respectively, $J=J^F$) for any ideal $J$ of $R$. 
\end{definition}

It is clear that $F$-regularity implies $F$-purity. We remark that for a commutative Noetherian ring of characteristic $p>0$, there are three different definitions of $F$-regularity, and whether they are equivalent in this generality remains open. Fortunately, by celebrated results of Lyubeznik and Smith \cite[Corollaries 4.3, 4.4]{LS99}, and Hochster \cite[page 237]{Ho07}, they indeed coincide for $\mathbb{N}$-graded rings. Thus we will not distinguish among the three in this article.

\begin{theorem}\label{thm:properties-F}
    The following statements hold for $\mathbb{N}$-graded rings.
    \begin{enumerate}
        \item Regular rings are $F$-regular.
        \item Direct summands of $F$-regular rings are $F$-regular.
        \item $F$-regular rings are normal domains.
        \item $F$-pure rings are reduced.
    \end{enumerate}
\end{theorem}
\begin{proof}
    For the first three parts, see \cite[Theorem 4.2]{HH94a}. The last part follows directly from~definitions.
\end{proof}

Checking $F$-regularity and $F$-purity is usually a difficult task, and the popular methods are Glassbrenner's and Fedder's criteria, recorded below in the case of $\mathbb{N}$-graded algebras. We recall that the ring $R$ is called \textit{$F$-finite} if the ring $R^{1/p}$, obtained by adjoining all $p$-th roots of elements in $R$, is a finitely generated $R$-module.

\begin{theorem}[{Glassbrenner's criterion, \cite[Theorem 3.1]{Gl96}}]\label{thm:Glassbrenner}
    Assume $R=S/I$ is $F$-finite and let $\mathfrak{m}$ denote the homogeneous maximal ideal of $S$. Then $R$ is $F$-regular if and only if for each $c$ of $R$ that is not in any minimal prime of $I$, we have $c(I^{[q]}:_S I)\nsubseteq \mathfrak{m}^{[q]}$ for some $q=p^e$.
\end{theorem}

\begin{theorem}[{Fedder's criterion, \cite[Theorem 1.12]{Fe83}}]\label{thm:Fedder}
    Assume $R=S/I$ is reduced and $F$-finite and let $\mathfrak{m}$ denote the homogeneous maximal ideal of $S$. Then $R$ is $F$-pure if and only if $(I^{[p]}:_S I)\nsubseteq \mathfrak{m}^{[p]}$.
\end{theorem}

In the special case when $I$ is generated by a regular sequence, we have
\[ (I^{[q]}:_S I) = (\omega^{q-1}) + I^{[q]},  \]
where $\omega$ denotes the product of minimal generators of $I$. Therefore, to show that $I$ in this case defines an $F$-pure (respectively, $F$-regular) ring, it is necessary and sufficient to show that $\omega^{p-1} \notin \mathfrak{m}^{[p]}$ (respectively, that for each non-zerodivisor $c$ of $R$, we have $c\omega^{q-1} \notin \mathfrak{m}^{[q]}$ for some $q=p^e$). 

Another popular technique in showing $F$-purity and $F$-regularity is deformation. Although these two properties do not deform in general (\cite[Theorem~1.1]{Si99} and \cite[Example~4.8]{Fe83}), they do deform for Gorenstein rings:

\begin{theorem}[{\cite[Theorem 4.2]{HH94a}, \cite[Theorem 3.4]{Fe83}}]\label{thm:fregdeform}
    Let $f$ be an $R$-regular element. If $R$ is Gorenstein and $R/fR$ is $F$-regular (respectively, $F$-pure), then $R$ is $F$-regular (respectively, $F$-pure).
\end{theorem}

It turns out that it is sufficient to verify Glassbrenner's criterion for just one certain non-zerodivisor $c$, and such an element is called a \textit{test element}. There are many test elements. In this article, the lemma below is how we will obtain test elements.

\begin{lemma}[{\cite[Corollary 2.6]{TC-Commutator}}]\label{lem:testelement}
    Assume that $R$ is an $F$-finite and $F$-pure ring. Then any non-zerodivisor $c$ such that $R_c$ is regular is a test element.
\end{lemma}

\subsection{Permanental ideals}

We recall the following result about determinantal ideals.

\begin{theorem}[{\cite[Theorem 7.14]{HH94b} and \cite[Theorem 2.2]{CSV24}}]\label{thm:det}
    Let $\Bbbk$ be a field of positive characteristic $p>0$, and $X$ a generic or symmetric matrix of indeterminates. Then $\Bbbk[X]/I_t(X)$ is $F$-regular for any integer~$t$.
\end{theorem}

Algebras defined by determinantal ideals are among the first examples of $F$-regular rings. We remark that by definitions, permanental ideals and determinantal ideals coincide when the base field $\Bbbk$ is of characteristic $2$. We will thus focus on the case where the two classes do not coincide, i.e., when $\characteristic \Bbbk \neq 2$.  

It turns out that in this case, permanental ideals contain many monomials. We recall the two lemmas that will be used later regarding $P_2(X)$. These are from \cite{LS00}. We note that the authors in this paper focus on the case where $X$ is a generic matrix, but their proofs in fact work for any matrix. We remark that Kirkup has a similar result (\cite[Corollary 6]{Kirkup}) for $P_t(X)$ in general.

\begin{lemma}[{\cite[Lemma 2.1]{LS00}}]\label{lem:three-element-in-permanent}
     Let $X$ be an $m\times n$ matrix of indeterminates where $m,n\geq 2$. Then the ideal $P_2(X)$ contains all products of three entries of $X$, taken from three distinct columns and two distinct rows (if $n\geq 3$), or from two distinct columns and three distinct rows (if $m\geq 3$).
\end{lemma}

\begin{lemma}[{\cite[Lemma 2.2]{LS00}}]\label{lem:three-element-in-permanent-2}
    Let $X$ be an $m\times n$ matrix of indeterminates where $m,n\geq 3$. Then the ideal $P_2(X)$ contains all products of the form $x_{i_1j_1}^2x_{i_2j_2}x_{i_3j_3}$ with distinct $i_1,i_2,i_3$ and distinct $j_1,j_2,j_3$.
\end{lemma}

\section{The $F$-regularity of the hypersurface defined by the permanent of a square generic/symmetric/Hankel matrix}\label{sec:square}

The goal of this section is to show that the hypersurface $\Bbbk[X]/P_n(X)$ is $F$-regular whenever $\Bbbk$ is a field of positive characteristic $p>2$ and $X$ an $n\times n$ generic/symmetric/Hankel matrix. In fact, we will prove the result for a Hankel matrix first, and analogs for generic/symmetric matrices later as a corollary.

To show the $F$-regularity of a ring, it suffices to do so for a faithfully flat extension \cite[Proposition~4.12]{HH-BrianconSkoda}. Therefore, for the rest of this section, we can assume that $\Bbbk$ is algebraically closed. Set 
\[
Z_n=\begin{pmatrix}
    z_1 & z_2 & \cdots & z_n\\
    z_2 & z_3 & \cdots & z_{n+1}\\
    \vdots & \vdots &\ddots &\vdots\\
    z_{n} & z_{n+1} & \cdots & z_{2n-1}
\end{pmatrix}.
\]
Then $Z_n$ is an $n\times n$ Hankel matrix of indeterminates. We start with proving that $\perm(Z)$ is irreducible, for which we need a lemma.

\begin{lemma}\label{lem:perm-not-contain-nice-monomials}
    The polynomial $\perm(Z_n)$ does not contain a monomial term of the form $z_{n+1}^{n-k} z_n^{k}$ for any $k\in [0,\ n-1]$.
\end{lemma}
\begin{proof}
    Suppose otherwise that $\perm(Z_n)$ contains a monomial term of the form $z_{n-1}^{n-k} z_n^{k}$ for some $k\in [0,\ n-1]$. Then this term survives after substitute $0$ into the variables $z_i$ for any $i\in [0,\ n) \cup (n+1,\ 2n-1]$. After the substitutions, $\perm(Z_n)$ becomes
    \[
    \perm \begin{pmatrix}
    0 & 0 & \cdots & z_n\\
    0 & 0 & \cdots & z_{n+1}\\
    \vdots & \vdots &\ddots &\vdots\\
    z_{n} & z_{n+1} & \cdots & 0
\end{pmatrix} = z_n^n,
    \]
    a contradiction. Thus the result follows.
\end{proof}

\begin{lemma}\label{lem:square-Hankel-domain}
    For any $n\geq 1$, the ring $\Bbbk[Z]/P_n(Z_n)$ is a domain.
\end{lemma}

\begin{proof}
    The ideal $P_n(Z_n)$ is generated by the permanent of $Z_n$. Thus the problem becomes showing that $\perm(Z_n)$ is irreducible, which we will show with Eisenstein's criterion (see, e.g., \cite[Chapter 9, Proposition 13]{DummitFoote}) for large $n$. If $n=1$, then $\perm(Z_n)=z_1$ is irreducible. If $n=2$, then $\perm(Z_n)=z_2^2+z_1z_3= \det \begin{pmatrix}
        z_2 & -z_{1}\\
        z_3&z_2
    \end{pmatrix}$ is irreducible, as desired. 
    
    Now assume that $n\geq 3$. Set 
    \begin{align*}
        R&\coloneqq\Bbbk[z_i\mid i\in [1,n) \cup (n,2n-1]],\\
        P&\coloneqq (z_i\mid i\in [1,n)\cup (n+1,2n-1]),\\
        \perm (Z_n) &\coloneqq \sum_{i=0}^n a_iz_n^i,
    \end{align*}
    where $a_i\in R$ for each $i\in [0,n]$. It is clear that $P$ is a prime ideal of $R$, and $\perm(Z_n)$ is a monic polynomial, i.e., $a_n=1$, in the polynomial ring $R[z_n]$. By Eisenstein's criterion, it suffices to show that $a_i\in P$, for each $i\in [0,\ n-1]$, and $a_0\notin P^2$. 
    
    Indeed, it is clear that, for each $i\in [0,\ n-1]$, the polynomial $a_i$ is homogeneous, and is not a unit. Due to Lemma~\ref{lem:perm-not-contain-nice-monomials}, $a_i$ does not contain a power of $z_{n+1}$ as a monomial term, and thus, must be in $P$. This verifies the first condition of the criterion. It remains to show that $a_0\notin P^2$. Observe that $a_0$ contains a summand that is a multiple of $z_1z_{n+1}^{n-1}$. As there is exactly one way to obtain such a monomial, $z_1z_{n+1}^{n-1}$ itself is a summand of $a_0$. Since $P^2$ is a monomial ideal and $z_1z_{n+1}^{n-1}\notin P^2$, we have $a_0\notin P^2$, as desired.
\end{proof}

Due to this result, any nonzero element of $\Bbbk[Z_n]/P_n(Z_n)$ is a nonzero divisor. To prove that this ring is $F$-regular, we now need a test element.  Consider the polynomial $\frac{\partial \perm(Z_n)}{z_{2n-1}}$ in the singular locus of $\Bbbk[Z_n]/P_n(Z_n)$ (see, e.g., \cite[Corollary 16.20]{Eisenbud}). We notice that this polynomial is exactly the permanent of $Z_{n-1}$. For this reason, we set $f_n\coloneqq \perm(Z_n)$. To use Glassbrenner's criterion, we will need to solve some ideal membership problem. For a matrix $Z$ and two integers $i$ and $j$, let $Z^i_j$ be the matrix $Z$ without the $i$-th column and $j$-th row. In particular, we have $(Z_n)^n_n=Z_{n-1}$.

\begin{lemma}\label{lem:ideal-membership-square-Hankel}
    Assume $p>2$. Then for any $n\geq 1$, we have \[
    \left(f_{n-1}f_{n}^{p-1}\right) \left(\prod_{i=1}^{n-1} z_{2i+1} \right)\left(\prod_{i=1}^{n-1} z_{2i}\right)^{p-3} = (-1)^{n+1} \left( \prod_{i=1}^{2n-1} z_i \right)^{p-1} \text{ modulo } (z_1^p, z_2^p,\dots, z_{2n-1}^p).
    \]
\end{lemma}
\begin{proof}
    In this proof, we will consider all elements modulo $(z_1^p,z_2^p,\dots, z_{2n-1}^p)$. Set
    \[
    F_n\coloneqq \left(f_{n-1}f_{n}^{p-1}\right) \left(\prod_{i=1}^{n-1} z_{2i+1} \right)\left(\prod_{i=1}^{n-1} z_{2i}\right)^{p-3}.
    \]
    If $n=1$, then $F_n= f_1^{p-1} = z_1^{p-1}$, as desired. We can now assume that $n\geq 2$, and by induction, assume that $F_{n-1} = (-1)^{n} \left( \prod_{i=1}^{2n-3} z_i \right)$. We have
    \begin{align*}
        \deg F_n &= \deg \left(f_{n-1}f_{n}^{p-1}\right) \left(\prod_{i=1}^{n-1} z_{2i+1} \right)\left(\prod_{i=1}^{n-1} z_{2i}\right)^{p-3} \\
        &= (n-1) +\big( n(p-1)\big) + (n-1) + \big((n-1)(p-3)\big) \\
        &= (2n-1)(p-1). 
    \end{align*}
    Since there are exactly $2n-1$ variables, $F_n$ must be a multiple of $\left( \prod_{i=1}^{2n-1} z_i \right)$. In other words, we can ignore all other monomial terms that appear in the distribution of $F_n$. First we focus on $z_{2n-1}$, and apply the expansion formula using the last row:
    \begin{align*}
        F_n &= \left(f_{n-1}\right)\left(f_{n}^{p-1}\right) \left(\prod_{i=1}^{n-1} z_{2i+1} \right)\left(\prod_{i=1}^{n-1} z_{2i}\right)^{p-3}\\
        &= \left(f_{n-1}\right)\left( \sum_{i=1}^n z_{n+i-1} \perm\big(Z_n\big)_n^i  \right)^{p-1} \left(\prod_{i=1}^{n-1} z_{2i+1} \right)\left(\prod_{i=1}^{n-1} z_{2i}\right)^{p-3}\\
        &= \left(f_{n-1}\right)\left(  z_{2n-1} f_{n-1} + G \right)^{p-1} \left(\prod_{i=1}^{n-1} z_{2i+1} \right)\left(\prod_{i=1}^{n-1} z_{2i}\right)^{p-3},
    \end{align*}
    where $G\coloneqq \sum_{i=1}^{n-1} z_{n+i-1} \perm\big(Z_n\big)_n^i $ is a polynomial that does not involve $z_{2n-1}$. Note that $f_{n-1}$ also does not involve $z_{2n-1}$. Thus, there is only one term that matters in $\left(  z_{2n-1} f_{n-1} + G \right)^{p-1}$. We then have
    \begin{align*}
        F_n&=  \left(f_{n-1}\right)\left( \binom{p-1}{p-2} \left( z_{2n-1} f_{n-1}\right)^{p-2} (G) \right) \left(\prod_{i=1}^{n-1} z_{2i+1} \right)\left(\prod_{i=1}^{n-1} z_{2i}\right)^{p-3}\\
        &=\left(-z_{2n-1}^{p-1}\right) \left(f_{n-1}^{p-1}\right)\left(  G \right) \left(\prod_{i=1}^{n-2} z_{2i+1} \right)\left(\prod_{i=1}^{n-1} z_{2i}\right)^{p-3}.
    \end{align*}
    Next we look at the variable $z_{2n-2}$. In the above product, except for $z_{2n-2}^{p-3}$, the only polynomial that can contribute $z_{2n-2}^2$ is $G$. We notice that $G$ is the sum of products of a variable and a subpermanent of $Z_n$, and $Z_n$ has exactly two $z_{2n-2}$ as entries. Thus, there is exactly one term in $G$ that matters:
    \begin{align*}
        F_n&=\left(-z_{2n-1}^{p-1}\right) \left(f_{n-1}^{p-1}\right)\left( \sum_{i=1}^{n-1} z_{n+i-1} \perm\big(Z_n\big)_n^i   \right) \left(\prod_{i=1}^{n-2} z_{2i+1} \right)\left(\prod_{i=1}^{n-1} z_{2i}\right)^{p-3}\\
        &= \left(-z_{2n-1}^{p-1}\right) \left(f_{n-1}^{p-1}\right)\left( z_{2n-2} \perm \left( \big(Z_n\big)_n^{n-1} \right)   \right) \left(\prod_{i=1}^{n-2} z_{2i+1} \right)\left(\prod_{i=1}^{n-1} z_{2i}\right)^{p-3},
    \end{align*}
    and as mentioned, the only term that matters in $\perm \left( \big(Z_n\big)_n^{n-1} \right)$ is the one that contains $z_{2n-2}$:
    \begin{align*}
        F_n&=\left(-z_{2n-1}^{p-1}\right) \left(f_{n-1}^{p-1}\right)\left( z_{2n-2}^2 \perm \left( \big(Z_n\big)_n^{n-1} \right)^{n-1}_{n-1}   \right) \left(\prod_{i=1}^{n-2} z_{2i+1} \right)\left(\prod_{i=1}^{n-1} z_{2i}\right)^{p-3}\\
        &=\left(-z_{2n-1}^{p-1}\right) \left(f_{n-1}^{p-1}\right)\left( z_{2n-2}^2 f_{n-2}  \right) \left(\prod_{i=1}^{n-2} z_{2i+1} \right)\left(\prod_{i=1}^{n-1} z_{2i}\right)^{p-3}\\
        &=-\left(z_{2n-2}z_{2n-1}\right)^{p-1} \left(f_{n-2} f_{n-1}^{p-1}\right) \left(\prod_{i=1}^{n-2} z_{2i+1} \right)\left(\prod_{i=1}^{n-2} z_{2i}\right)^{p-3}.
    \end{align*}
    Observe that $F_{n-1}$ has appeared. Thus we can use the induction hypothesis:
    \begin{align*}
        F_n &=  -\left(z_{2n-2}z_{2n-1}\right)^{p-1} (F_{n-1})\\
        &=-\left(z_{2n-2}z_{2n-1}\right)^{p-1}  \left( (-1)^{n} \left( \prod_{i=1}^{2n-3} z_i \right)^{p-1} \right)\\
        &= (-1)^{n+1} \left( \prod_{i=1}^{2n-1} z_i \right)^{p-1},
    \end{align*}
    as desired.
\end{proof}

We are in a position to prove the main theorem of the section.

\begin{theorem}\label{thm:square-Hankel-F-regular}
    Let $\Bbbk$ be a field of positive characteristic $p>2$ and $Z_n$ an $n\times n$ Hankel matrix of indeterminates. Then the ring $\Bbbk[Z]/P_n(Z_n)$ is $F$-regular.
\end{theorem}

\begin{proof}
    The ideal $P_n(Z_n)$ is generated by $f_n$. Consider the diagonal monomial order $z_1>z_2>\cdots>z_{2n-1}$. With respect to this monomial order, we have $\initial(f_n)=\prod_{i=1}^{n}z_{2i-1}$. We then have \[\initial(f_n^{p-1}) = \initial(f_n)^{p-1} =\prod_{i=1}^{n}z_{2i-1}^{p-1} \notin (z_1^p,z_2^p,\dots, z_{2n-1}^p),\] 
    and thus $\Bbbk[Z_n]/(f_n)$ is $F$-pure by Fedder's criterion (Theorem~\ref{thm:Fedder}). We will show that $\Bbbk[Z_n]/(f_n)$ is in fact $F$-regular using Glassbrenner's criterion. By the Jacobian criterion (see, e.g., \cite[Corollary 16.20]{Eisenbud}), the polynomial $\frac{\partial f_n}{\partial z_{2n-1}}=f_{n-1}$ is in the singular locus of $\Bbbk[Z_n]/(f_n)$. Couple this with the fact that $\Bbbk[Z_n]/(f_n)$ is a domain (Lemma~\ref{lem:square-Hankel-domain}), the polynomial $f_{n-1}$ is a test element by Corollary~\ref{lem:testelement}. By Glassbrenner's criterion, to show that $\Bbbk[Z_n]/(f_n)$ is $F$-regular, it suffices to show that $f_{n-1}f_n^{p-1}\notin (z_1^p,z_2^p,\dots, z_{2n-1}^p)$, which follows immediately from Lemma~\ref{lem:ideal-membership-square-Hankel}, as~desired.
\end{proof}

In the case of a square generic or symmetric matrix, an analog follows as a corollary.

\begin{theorem}\label{thm:square-generic-symmetric-F-regular}
    Let $\Bbbk$ be a field of positive characteristic $p>2$ and $X$ an $n\times n$ generic or symmetric matrix of indeterminates. Then the ring $\Bbbk[X]/P_n(X)$ is $F$-regular. 
\end{theorem}

\begin{proof} 
    We claim that the generic and symmetric permanental hypersurfaces are a deformation of the corresponding Hankel one.  Assume that $X$ is generic. We will use a specialization in \cite[page 112]{CONCA2018111}. Let $\mathbf{x}$ be the set of linear forms
	\[
	\{x_{1k}-x_{1+l,\ k-l}\mid 1\leq l< k\leq n\} \cup \{x_{k,\ n}-x_{l,\ n+k-l}\mid 2\leq k<l< n  \}
	\]
	of $\Bbbk[X]$. We then have the following isomorphism
	\[
	\Bbbk[X]/ \big( P_n(X)+ ( \mathbf{x})\big) \cong \Bbbk[Z_n]/(P_n(Z_n)),
	\]
	and 
	\begin{align*}
		\dim(\Bbbk[X]/P_n(X)) -\dim( \Bbbk[Z_n]/P_n(Z_n)) &= (n^2-1)-(2n-2)\\
		&=(n-1)^2=|\mathbf{x}|.
	\end{align*}
	In particular, this means that the ring $\Bbbk[Z_n]/P_n(Z_n)$, where $Z_n$ is a Hankel matrix, is isomorphic to $\Bbbk[X]/P_n(X)$, where $X$ is a generic matrix, modulo a regular sequence. The claim for the generic permanental hypersurface then follows. Similar arguments apply to the symmetric case, where the specialization is by setting the variables on each antidiagonal to be equal.
    Since $F$-regularity deforms for Gorenstein rings (\cite[Theorem 4.2]{HH94a}), the result then follows from Theorem~\ref{thm:square-Hankel-F-regular}. 
\end{proof}

By a theorem of Smith (\cite[Theorem 4.3]{Sm97}), we obtain the following.

\begin{corollary}\label{cor:square-rational-singularities}
    Let $\Bbbk$ be a field of characteristic $0$ and $X$ an $n\times n$ generic/symmetric/Hankel matrix of indeterminates. Then the ring $\Bbbk[X]/P_n(X)$ has rational singularities.
\end{corollary}

Unlike their determinantal counterparts, permanental ideals are rarely prime, and hence rarely define $F$-regular rings in prime characteristic $p\neq 2$. We record some results below that support this argument:
\begin{enumerate}
    \item if $X$ is a generic/symmetric/Hankel matrix, then $P_2(X)$ is prime if and only if $X$ is a $2\times 2$ matrix (Lemma~\ref{lem:three-element-in-permanent} and Theorem~\ref{thm:square-generic-symmetric-F-regular});
    \item if $X$ is a generic matrix, then $P_3(X)$ is prime if and only if $X$ is a $3\times 3$ matrix (\cite[Corollary~12 and Corollary~17]{Kirkup} and Theorem~\ref{thm:square-generic-symmetric-F-regular});
    \item if $X$ is a Hankel matrix, then $P_3(X)$ is prime if and only if $X$ is a $3\times 3$ matrix (\cite{GG06}).
\end{enumerate}

With this evidence, we propose the following conjecture:

\begin{conjecture}
    Let $X$ be an $m\times n$ generic/symmetric/Hankel matrix and $\Bbbk$ a field of characteristic $p\neq 2$. Then the ideal $P_t(X)$ of $\Bbbk[X]$ is prime if and only if $t=m=n$.
\end{conjecture}

\section{The $F$-singularities of algebras defined by $2\times 2$ permanents of generic matrices}\label{sec:generic}

In this section we will determine the $F$-singularities of the coordinate rings $\Bbbk[X]/\sqrt{P_2(X)}$ where $\Bbbk$ is a field of positive characteristic $p> 2$ and $X$ is a $m\times n$ generic matrix of indeterminates where $m,n\geq 2$. We remark that since $F$-pure rings are reduced (Theorem \ref{thm:properties-F}), the ring $\Bbbk[X]/P_2(X)$ would not even be $F$-pure if $P_2(X)$ is not radical. Thus we study $\Bbbk[X]/\sqrt{P_2(X)}$. The algebraic properties of these rings were studied in detail in \cite{LS00}. We recall the following result that will help us showing the $F$-purity of the algebra defined by an ideal by looking at those defined by its minimal primes.

\begin{lemma}[{\cite[Lemma 5.5]{PTW23}}]\label{lem:intersection}
    Let $I$ and $J$ be ideals of a polynomial ring $S$ in finitely many variables over a field of prime characteristic $p>0$. Then
	\[
	(I^{[p]}:_S I) \cap (J^{[p]}:_S J) \subseteq (I\cap J)^{[p]}:_S (I\cap J).
	\]
\end{lemma}

Essentially, to show that  $\Bbbk[X]/\sqrt{P_2(X)}$ is $F$-pure, it suffices to find an element $f$ such that for any minimal prime $P$ of $P_2(X)$, we have $f\in (P^{[p]}:_{\Bbbk[X]} P) \setminus \mathfrak{m}^{[p]}$, where $\mathfrak{m}$ denote the homogeneous maximal ideal of $\Bbbk[X]$. We recall a result about the minimal primes of $P_2(X)$.

\begin{lemma}[{\cite[Theorem 4.1]{LS00}}]\label{lem:generic-minimal-primes}
    An ideal $P$ is a minimal prime of $P_2(X)$ if and only if one of the following holds:
    \begin{enumerate}
        \item $P$ is generated by the permanent of one $2\times 2$ submatrix of $X$ and all the entries
        of $X$ outside of this submatrix;
        \item $P$ is generated by all the indeterminates in $m - 1$ of the rows of $X$ (if $n\geq 3$);
        \item $P$ is generated by all the indeterminates in $n - 1$ of the columns of $X$ (if $m\geq 3$).
    \end{enumerate}
\end{lemma}

Next we identify the desired element $f$. Let $\Omega$ denote the set of all $2\times 2$ submatrices of $X$. For each $\omega\in \Omega$, set $\omega =\begin{pmatrix}
	    a_\omega&b_\omega\\
        c_\omega&d_\omega
	\end{pmatrix}$. We will also use $\omega$ to denote the set $\{a_\omega, b_\omega, c_{\omega}, d_{\omega}\}$. Set
	\[
	f\coloneqq \prod_{i,j=1}^{n} x_{ij}^{p-1} + \sum_{\omega \in \Omega}\left( \left( \prod_{x_{ij}\notin \omega } x_{ij}^{p-1} \right) \sum_{k=0}^{p-2} (-1)^k (a_\omega d_\omega)^{2p-2-k}(b_\omega c_\omega)^k\right).
	\]	

\begin{lemma}\label{lem:element-F-pure-generic}
    For each minimal prime $P$ of $P_2(X)$, we have $f\in (P^{[p]}:_{\Bbbk[X]} P)$.
\end{lemma}
\begin{proof}
    Set
	\begin{align*}
		g&\coloneqq \prod_{i,j=1}^{n} x_{ij}^{p-1}, \text{ and}\\
		f_{\omega}&\coloneqq\left( \left( \prod_{x_{ij}\notin \omega } x_{ij}^{p-1} \right) \sum_{k=0}^{p-2} (-1)^k (a_\omega d_\omega)^{2p-2-k}(b_\omega c_\omega)^k\right),
	\end{align*}
    for each $\omega\in \Omega$. Then $f=g+\sum_{\omega\in \Omega} f_\omega$. 	Knowing all the minimal primes $P$ over $P_2(X)$ from Lemma~\ref{lem:generic-minimal-primes}, we will show that $f\in (P^{[p]}:_{\Bbbk[X]} P)$ for each of them.

    \textbf{Case 1:} $P$ is generated by the permanent of a submatrix $\omega'=\begin{pmatrix}
			w&x\\y&z
		\end{pmatrix}$ of $X$ and all the variables of $X$ outside of this submatrix.  It is clear that these generators of $P$ form a regular sequence on $\Bbbk[X]$. Hence
		\begin{align*}
			(P^{[p]}:_{\Bbbk[X]} P)&=\left( (wz+xy)^{p-1}\prod_{ x_{ij}\notin \omega' } x_{ij}^{p-1}\right) + P^{[p]}\\
			&=\left( (wz+xy)^{p-1}\prod_{ x_{ij}\notin \omega' } x_{ij}^{p-1} \right)+ \big((wz+xy)^p\big) +\big(x_{ij}^p\mid x_{ij}\neq w,x,y,z\big).
		\end{align*} 
		We have
		\begin{align*}
			g+f_{\omega'}&=\prod_{i,j=1}^{n} x_{ij}^{p-1} + \left( \prod_{x_{ij}\notin \omega'} x_{ij}^{p-1} \right) \sum_{k=0}^{p-2} (-1)^k (wz)^{2p-2-k}(xy)^k\\
			&=\left( \prod_{x_{ij}\notin \omega' } x_{ij}^{p-1} \right) \left(wz\right)^{p-1} \sum_{k=0}^{p-1} (-1)^k (wz)^{p-1-k}(xy)^k
		\end{align*}
        It is known that $\binom{p-1}{k}= (-1)^k \text{ (mod } p)$. Thus
        \begin{align*}
            		g+f_{\omega'}	&=\left(wz\right)^{p-1} \left( \prod_{x_{ij}\notin \omega'  } x_{ij}^{p-1} \right) \sum_{k=0}^{p-1} \binom{p-1}{k} (wz)^{p-1-k}(xy)^k\\
			&=\left(wz\right)^{p-1} \left( \prod_{x_{ij}\notin \omega'  } x_{ij}^{p-1} \right)(wz+xy)^{p-1}
        \end{align*}
		is in $(P^{[p]}:_{\Bbbk[X]} P)$. On the other hand, consider any submatrix $\omega\neq \omega'= \begin{pmatrix}
		w&x\\y&z
		\end{pmatrix}$ of $X$. This is equivalent to either $a_\omega\notin \omega'$ or $d_\omega\notin \omega'$. Moreover, each monomial summand of the polynomial
		\[
		f_{\omega}=\left( \prod_{x_{ij}\notin \omega } x_{ij}^{p-1} \right) \sum_{k=0}^{p-2} (-1)^k (a_\omega d_\omega)^{2p-2-k}(b_\omega c_\omega)^k
		\]
		contains either $a_\omega^p$ or $d_\omega^p$. Thus $f_{\omega}$ is in $(P^{[p]}:_{\Bbbk[X]} P)$ for each $\omega \neq \omega'$. Therefore
		\[
		f=g+\sum_{\omega \in \Omega} f_{\omega}
		\]
		is in $(P^{[p]}:_{\Bbbk[X]} P)$, as desired.

        \textbf{Case 2:} $P$ is generated by all the variables in some $n-1$ of the columns of $X$ (in this case we have $m>2$). By symmetry, we can assume that $P$ is generated by all the variables in the first $n-1$ columns of $X$. These generators of $P$ form a regular sequence on $\Bbbk[X]$. Hence
		\[
		(P^{[p]}:_{\Bbbk[X]} P) = \prod_{i\in [1,n],\ j\in [1,n-1] } x_{ij}^{p-1} + P^{[p]}.
		\]
		It is clear that $g$ is in $(P^{[p]}:_{\Bbbk[X]} P)$. On the other hand, consider any submatrix $\omega\in \Omega$. Then the variable $a_\omega$ is in the first $n-1$ columns of $X$. Thus, $f_{\omega}$, being divisible by $a_\omega^p$, is also in $(P^{[p]}:_{\Bbbk[X]} P)$. Therefore, 
		\[
		f=g+\sum_{\omega \in \Omega} f_{\omega}
		\]
		is in $(P^{[p]}:_{\Bbbk[X]} P)$, as desired.

        \textbf{Case 3:} $P$ is generated by all the variables in some $m-1$ of the rows of $X$ (in this case we have $n>2$). The proof is similar to that of the previous case by symmetry. This concludes the~proof.
\end{proof}

We are now ready to prove the main theorem of this section.

\begin{theorem}\label{thm:generic-Fpure}
    Let $\Bbbk$ be a field of positive characteristic $p>2$ and $X$ an $m\times n$ generic  matrix of indeterminates where $m,n\geq 2$. Then
	\begin{enumerate}
		\item the ring $\Bbbk[X]/\sqrt{P_2(X)}$ is $F$-pure;
		\item the ring $\Bbbk[X]/\sqrt{P_2(X)}$ is $F$-regular if and only if $m=n=2$.
	\end{enumerate}
\end{theorem}

\begin{proof}
	If $(m,n)=(2,2)$, then     by Theorem~\ref{thm:square-generic-symmetric-F-regular}, $\Bbbk[X]/\sqrt{P_2(X)}= \Bbbk[X]/P_2(X)$ is $F$-regular, as desired. For the rest of the proof, we will assume that $(m,n)\neq (2,2)$.
	
	By Lemma~\ref{lem:generic-minimal-primes}, $\sqrt{P_2(X)}$ is not prime, and thus the corresponding ring $\Bbbk[X]/\sqrt{P_2(X)}$ is not $F$-regular by Theorem~\ref{thm:properties-F}. The statement (2) then follows.
    
	We will now prove (1). It is clear that  $f= \prod_{i,j=1}^n x_{ij}^{p-1}$ modulo $\mathfrak{m}^{[p]}$, where $\mathfrak{m}$ denotes the homogeneous maximal ideal of $\Bbbk[X]$, and hence $f\notin\mathfrak{m}^{[p]}$. By Lemma~\ref{lem:element-F-pure-generic} and Lemma~\ref{lem:intersection}, we have $f\in (P_2(X)^{[p]}:_{\Bbbk[X]} P_2(X))$. Thus statement (1) follows from Fedder's criterion (Theorem~\ref{thm:Fedder}).
\end{proof}

It is natural to ask whether $\Bbbk[X]/\sqrt{P_t(X)}$ is $F$-pure for $t>2$. Permanental ideals become much more complicated even for $t=3$, where we do not even have a complete list of their minimal primes (a partial list in this case can be found in \cite{Kirkup}). It is also not known in general when $P_t(X)$, or even $P_3(X)$, is radical. In a recent article, specifically \cite[Theorem~3.18]{BCMV25}, the authors showed that $P_t(X)$, where $X$ is a $t\times(t+1)$ generic matrix, defines a complete intersection for $t\in [2,4]$. Thus in these three cases, the question of $F$-purity is equivalent to an ideal membership problem. In the case $t=2$, we already have an answer in this section. Using Macaulay2, we propose the following conjecture on $P_3(X_{3\times 4})$, which we verified for $\characteristic \Bbbk\leq 31$, and leave $P_4(X_{4\times 5})$ for speculation.

\begin{conjecture}
    Let $X$ be a $3\times 4$ generic matrix of indeterminates, and assume that $\characteristic\Bbbk >2 $. Then $\Bbbk[X]/P_3(X)$ is $F$-pure if and only if $p\equiv 1 \text{ mod }6$.
\end{conjecture}

\section{The $F$-singularities of algebras defined by $2\times 2$ permanents of symmetric matrices}\label{sec:sym}

In this section we will determine the $F$-singularities of the coordinate rings $\Bbbk[Y]/\sqrt{P_2(Y)}$ where $\Bbbk$ is a field of positive characteristic $p> 2$ and $Y$ is a $n\times n$ symmetric matrix of indeterminates where $n\geq 2$. Our proofs will mirror those in the last section. Thus we will use the same symbols for many notations. We are not to confuse the notations in this section with those in the previous section.

We start with a description of the minimal primes of the ideal $P_2(Y)$. For each pair of integers $1\leq u<v\leq n$, set $\omega_{uv}\coloneqq\begin{pmatrix}
    y_{uu}&y_{uv}\\
	y_{uv}&y_{vv}
\end{pmatrix}$, and let $P_{uv}$ denote the ideal generated by $\perm(\omega_{uv})$ and 
and all the entries of $Y$ outside of this submatrix. We then have the following.

\begin{lemma}[{\cite[Theorem~4.1]{Trung-symmetric}}]\label{lem:symmetric-minimal-primes}
    An ideal $P$ is a minimal prime of $P_2(Y)$ if and only if $P=P_{uv}$ for some $1\leq u<v\leq n$. 
\end{lemma}

Next we identify the magic element $f$ needed to use the Fedder's criterion.  We will also use $\omega_{uv}$ to denote the set $\{y_{uu}, y_{uv}, y_{vv}\}$. Set
\begin{align*}
    \!\begin{multlined}[t][.3\displaywidth]
	f\coloneqq (-1)^{\frac{p-1}{2}}  \prod_{1\leq i\leq j\leq n} y_{ij}^{p-1} + \sum_{1\leq i<j\leq n}\left( \left( \prod_{y\notin \omega_{ij}  } y^{p-1} \right) \sum_{k=0}^{\frac{p-3}{2}} (-1)^k (y_{ii}y_{jj})^{\frac{3(p-1)}{2}-k}(y_{ij})^{2k}\right)\\
	+ \sum_{1\leq i<j\leq n}\left( \left( \prod_{y\notin \omega_{ij}  } 
    y^{p-1} \right) \sum_{k=\frac{p+1}{2}}^{p-1} (-1)^k (y_{ii}y_{jj})^{\frac{3(p-1)}{2}-k}(y_{ij})^{2k}\right).
    \end{multlined}
\end{align*}

\begin{lemma}\label{lem:element-F-pure-symmetric}
    For any $1\leq u<v\leq n$, we have $f\in (P_{uv}^{[p]}:_{\Bbbk[Y]} P_{uv})$.
\end{lemma}
\begin{proof}
    Set
	\begin{align*}
		g&\coloneqq (-1)^{\frac{p-1}{2}} \prod_{1\leq i\leq j\leq n} y_{ij}^{p-1}, \text{ and} \\
		f_{ij}&\coloneqq \!\begin{multlined}[t][.3\displaywidth]
			\left( \prod_{y\notin \omega_{ij} } y^{p-1} \right) \sum_{k=0}^{\frac{p-3}{2}} (-1)^k (y_{ii}y_{jj})^{\frac{3(p-1)}{2}-k}(y_{ij})^{2k}
			\\+ \left( \prod_{y\notin \omega_{ij} } y^{p-1} \right) \sum_{k=\frac{p+1}{2}}^{p-1} (-1)^k (y_{ii}y_{jj})^{\frac{3(p-1)}{2}-k}(y_{ij})^{2k},
		\end{multlined}
	\end{align*}
    for each pair of integers $1\leq i<j\leq n$. Then $f=g+\sum_{1\leq i<j\leq n} f_{ij}$.	 We recall the generators of $P_{uv}$:
	\[
	P_{uv}=(y_{uu}y_{vv}+y_{uv}^2)+(y \mid y\notin \omega_{uv}).
	\] 
	It is clear that these generators of $P_{uv}$ form a regular sequence on $\Bbbk[Y]$. Hence
		 \begin{align*}
			(P_{uv}^{[p]}:_{\Bbbk[Y]} P_{uv})&= (y_{uu}y_{vv}+y_{uv}^2)^{p-1}\prod_{ y\notin \omega_{uv} } y^{p-1} + P_{uv}^{[p]}\\
			&=\!\begin{multlined}[t][.3\displaywidth]
				\left( (y_{uu}y_{vv}+y_{uv}^2)^{p-1}\prod_{ y\notin \omega_{uv}} y^{p-1}   \right)+ \left((y_{uu}y_{vv}+y_{uv}^2)^p\right) +\left(y^p\mid y\notin \omega_{uv}\right).
			\end{multlined}
		\end{align*} 
		Thus the element
		\begin{align*}
			g+f_{uv}&=(-1)^{\frac{p-1}{2}} \prod_{1\leq i\leq j\leq n} y_{ij}^{p-1} + \!\begin{multlined}[t][.3\displaywidth]
				\left( \prod_{y\notin \omega_{uv}  } y^{p-1} \right) \sum_{k=0}^{\frac{p-3}{2}} (-1)^k (y_{uu}y_{vv})^{\frac{3(p-1)}{2}-k}(y_{uv})^{2k}
				\\+ \left( \prod_{y\notin \omega_{uv} } y^{p-1} \right) \sum_{k=\frac{p+1}{2}}^{p-1} (-1)^k (y_{uu}y_{vv})^{\frac{3(p-1)}{2}-k}(y_{uv})^{2k}
			\end{multlined}\\
			&=\left( \prod_{y\notin \omega_{uv}  } y^{p-1} \right) \sum_{k=0}^{p-1} (-1)^k (y_{uu}y_{vv})^{\frac{3(p-1)}{2}-k}(y_{uv})^{2k}\\
			&=\left( \prod_{y\notin \omega_{uv}} y^{p-1} \right) (y_{uu}y_{vv})^{\frac{p-1}{2}} \sum_{k=0}^{p-1} \binom{p-1}{k} (y_{uu}y_{vv})^{p-1-k}(y_{uv})^{2k}
			\\
			&= (y_{uu}y_{vv})^{\frac{p-1}{2}} \left( \prod_{y\notin \omega_{uv}} y^{p-1} \right) (y_{uu}y_{vv}+y_{uv}^2)^{p-1}
		\end{align*}
		is in $(P_{uv}^{[p]}:_{\Bbbk[Y]} P_{uv})$. On the other hand, consider any pair of integers $1\leq i<j\leq n$ such that $(i,j)\neq (u,v)$. Then both $(y_{ii}y_{jj})^p$ and $y_{ij}^p$ are in $(P_{uv}^{[p]}:_{\Bbbk[X]} P_{uv})$. Moreover, each monomial summand of the polynomial
		\[
		f_{ij}=\!\begin{multlined}[t][.3\displaywidth]
			\left( \prod_{y\notin \omega_{ij}  } y^{p-1} \right) \sum_{k=0}^{\frac{p-3}{2}} (-1)^k (y_{ii}y_{jj})^{\frac{3(p-1)}{2}-k}(y_{ij})^{2k}\\ + \left( \prod_{y\notin \omega_{ij} } y^{p-1} \right) \sum_{k=\frac{p+1}{2}}^{p-1} (-1)^k (y_{ii}y_{jj})^{\frac{3(p-1)}{2}-k}(y_{ij})^{2k}
		\end{multlined}
		\]
		contains either $(y_{ii}y_{jj})^p$ or $y_{ij}^p$. Thus $f_{ij}$ is in $(P_{uv}^{[p]}:_{\Bbbk[Y]} P_{uv})$. Therefore
		\[
		f=g+\sum_{1\leq i< j\leq n} f_{ij}
		\]
		is in $(P_{uv}^{[p]}:_{\Bbbk[Y]} P_{uv})$, as desired.
\end{proof}

We are now ready to prove the main result of this section.

\begin{theorem}\label{thm:symmetric-Fpure}
    Let $\Bbbk$ be a field of positive characteristic $p>2$ and $Y$ an  $n\times n$  symmetric matrix of indeterminates where $n\geq 2$. Then
	\begin{enumerate}
		\item the ring $\Bbbk[Y]/\sqrt{P_2(Y)}$ is $F$-pure;
		\item the ring $\Bbbk[Y]/\sqrt{P_2(Y)}$ is $F$-regular if and only if $n=2$.
	\end{enumerate}
\end{theorem}

\begin{proof}
    If $n=2$, then    by Theorem~\ref{thm:square-generic-symmetric-F-regular}, $\Bbbk[Y]/\sqrt{P_2(Y)}= \Bbbk[Y]/P_2(Y)$ is $F$-regular, as desired. For the rest of the proof, we will assume that $n\neq 2$.

    By Lemma~\ref{lem:symmetric-minimal-primes}, $\sqrt{P_2(Y)}$ is not prime, and thus the corresponding ring $\Bbbk[Y]/\sqrt{P_2(Y)}$ is not $F$-regular by Theorem~\ref{thm:properties-F}. The statement (2) then follows.

    We will now prove (1). Let $\mathfrak{m}$ denote the homogeneous maximal ideal of $\Bbbk[Y]$. For degree reasons, we have  $f= \prod_{1\leq i<j\leq n} y_{ij}^{p-1}$ modulo $\mathfrak{m}^{[p]}$, and hence in particular, we have $f\notin\mathfrak{m}^{[p]}$. By Lemma~\ref{lem:element-F-pure-symmetric} and Lemma~\ref{lem:intersection}, we have $f\in (P_2(Y)^{[p]}:_{\Bbbk[Y]} P_2(Y))$. Thus statement (1) follows from Fedder's criterion (Theorem~\ref{thm:Fedder}).
\end{proof}

We end the section with a natural question of whether these results can be generalized.

\begin{question}
    Let $\Bbbk$ be a field of positive characteristic $p>2$ and $Y$ an  $n\times n$  symmetric matrix of indeterminates. For which $p$ and $n$ is the ring $\Bbbk[Y]/\sqrt{P_t(Y)}$ $F$-pure?
\end{question}

\section*{Acknowledgement} This paper is an extension on the author's Ph.D thesis \cite{TC}. The author was supported by the NSF grants DMS 1801285, 2101671, and 2001368, and the Infosys Foundation. The author would like to thank his advisor Anurag K. Singh for suggesting this problem and for the constant encouragement and helpful discussions.  The author would like to thank Vaibhav Pandey, Irena Swanson, and Uli Walther for many helpful suggestions. Part of this work was done while the author visited Purdue University. The author would like to thank the Department of Mathematics at Purdue, and especially Annie and Bill Giokas, for their hospitality. The author would like to thank the referees for their careful reading and helpful feedback.

\subsection*{Data availability statement} Data sharing does not apply to this article as no new data were created or
analyzed in this study.

\subsection*{Conflict of interest} The authors declare that they have no known competing financial interests or personal
relationships that could have appeared to influence the work reported in this paper.

\bibliographystyle{amsplain}
\bibliography{refs}
\end{document}